\documentclass[12pt,reqno]{amsart}

\usepackage{hyperref}
\usepackage[usenames]{color}
\usepackage[latin1]{inputenc} 
\usepackage{float}
\usepackage{bbm}

\usepackage[english]{babel} 

\usepackage{amssymb, amsmath}
\usepackage{geometry,graphicx}

\newcommand{\e}{\mathrm{e}}
\newcommand{\dif}{\mathrm{d}}
\newcommand{\leb}{\mathrm{L}}

\theoremstyle{plain}
\newtheorem{theorem}{Theorem}[section] 

\newtheorem{lemma}[theorem]{Lemma} 

\newtheorem{remark}[theorem]{Remark}

\newtheorem{example}[theorem]{Example}

\numberwithin{equation}{section} 
\title[]{Improving the approximation of the first and second order statistics of the response process to the random Legendre differential equation}

\author{J. Calatayud, J.-C. Cort\'{e}s, M. Jornet} 

\begin{document}

\maketitle

\begin{center}
\noindent
\address{Instituto Universitario de Matem\'{a}tica Multidisciplinar,\\
Universitat Polit\`{e}cnica de Val\`{e}ncia,\\
Camino de Vera s/n, 46022, Valencia, Spain\\
email: jucagre@alumni.uv.es; jccortes@imm.upv.es; marjorsa@doctor.upv.es
}
\end{center}

\begin{abstract}
In this paper, we deal with uncertainty quantification for the random Legendre differential equation, with input coefficient $A$ and initial conditions $X_0$ and $X_1$. In a previous study [Calbo G. et al, Comput. Math. Appl., 61(9), 2782--2792 (2011)], a mean square convergent power series solution on $(-1/\e,1/\e)$ was constructed, under the assumptions of mean fourth integrability of $X_0$ and $X_1$, independence, and at most exponential growth of the absolute moments of $A$. In this paper, we relax these conditions to construct an $\leb^p$ solution ($1\leq p\leq\infty$) to the random Legendre differential equation on the whole domain $(-1,1)$, as in its deterministic counterpart. Our hypotheses assume no independence and less integrability of $X_0$ and $X_1$. Moreover, the growth condition on the moments of $A$ is characterized by the boundedness of $A$, which simplifies the proofs significantly. We also provide approximations of the expectation and variance of the response process. The numerical experiments show the wide applicability of our findings. A comparison with Monte Carlo simulations and gPC expansions is performed. 
\textit{AMS Classification 2010: 34F05; 60H10; 60H35; 65C05; 65C60; 93E03.}
\end{abstract}

\section{Introduction} \label{introduction}

Random differential equations are differential equations in which randomness appears in the coefficients, forcing term, initial conditions and/or boundary conditions. The solution is a stochastic process that solves the differential equation in some probabilistic sense, usually in the sample path or $\leb^p$ sense. For a theoretical approach to random differential equations, see \cite{soong,strand}. Uncertainty quantification consists in calculating the main statistics of the response process to the stochastic system \cite{smith}. The main methods used to deal with uncertainty quantification for random differential equations are Monte Carlo simulations \cite{MC}, gPC based stochastic Galerkin technique \cite{depen,xiu}, finite difference schemes \cite{tawil,diffusion,diffusion2,Khodabin_MCM_2011,Nouri_Mediterr_2015,Nouri_Mediterr_2016,Khodabin_AD_2015,Jerez_Steklov,Soleymani_2015}, It\^o calculus \cite{oksendal} and $\leb^p$ calculus \cite{soong,L4jc,Licea_JCAM}. In the concrete case of second-order random linear differential equations, the Fr\"obenius method has been successfully used to deal with particular equations: Airy \cite{airy}, Hermite \cite{hermite}, Legendre \cite{previous}, Bessel \cite{bessel}, etc. In \cite{homotopy_Romanian,Khudair_1,Khudair_2},  homotopy, Adomian decomposition and differential transformations techniques, respectively, have been extended to the random scenario to solve some particular second-order random linear differential equations.

In this paper, we will deal with the random Legendre differential equation:
\begin{equation} \begin{cases} (1-t^2)\ddot{X}(t)-2t\dot{X}(t)+A(A+1)X(t)=0,\;|t|<1, \\
X(0)=X_0, \\
\dot{X}(0)=X_1. 
\end{cases} 
 \label{RDE}
\end{equation}
The coefficient $A$ is a non-negative random variable and the initial conditions $X_0$ and $X_1$ are random variables. All of them are defined in a common complete probability space $(\Omega,\mathcal{F},\mathbb{P})$. In \cite{previous}, the authors constructed a mean square convergent power series solution $X(t)$ to (\ref{RDE}) on $(-1/\e,1/\e)$ under certain assumptions on the random inputs $A$, $X_0$ and $X_1$. The goal of this article is to improve \cite{previous}: to weaken the hypotheses from \cite{previous}, to simplify the proofs significantly and to obtain an $\leb^p(\Omega)$ random power series solution on the whole domain $(-1,1)$, as in the deterministic counterpart of (\ref{RDE}). Numerical examples that could not be tackled via the hypotheses from \cite{previous} will be carried out in this paper, establishing a comparison with Monte Carlo simulations and a particular version of gPC expansions \cite{depen}. 

The structure of this paper is the following. In Section~\ref{legendreRDE}, we will review the techniques used in \cite{previous}. We will relax the assumptions from \cite{previous} and we will improve the conclusions of the results. In Section~\ref{approximationEV}, we will show how to approximate the expectation and variance of the response process, under no independence assumption. In Section~\ref{experiments}, we will perform a wide variety of examples and illustrate the potentiality of our findings by comparing the numerical results with Monte Carlo simulations and gPC expansions. Section~\ref{conclusions} will draw conclusions.

\section{Random Legendre differential equation} \label{legendreRDE}

In \cite{previous}, the authors constructed a mean square power series solution to the random Legendre differential equation (\ref{RDE}) on the time interval $(-1/\e,1/\e)$. The hypotheses assumed in \cite{previous} were that the absolute moments of $A$ increased at most exponentially, that is, there exist two positive constants $H$ and $M$ such that
\begin{equation}
\mathbb{E}[|A|^n]\leq H M^n,\quad n\geq n_0;
\label{EAn}
\end{equation}
that $A$ is independent of the initial conditions $X_0$ and $X_1$; and that $X_0,X_1\in\leb^4(\Omega)$. Hypothesis (\ref{EAn}) has been of constant use in the extant literature to study significant linear random differential equations via the Fr\"obenius method: \cite{airy,hermite,previous}. In \cite{previous}, the explicit solution to (\ref{RDE}) was obtained in the form of a random power series solution by means of the Fr\"{o}benius method:
\begin{equation}
X(t)=X_0\tilde{X}_1(t)+X_1\tilde{X}_2(t)
\label{Xt}
\end{equation}
for $|t|<1/\e$, where
\begin{equation}
\tilde{X}_1(t)=\sum_{m=0}^\infty \frac{(-1)^m}{(2m)!}P_1(m)t^{2m},\quad \tilde{X}_2(t)=\sum_{m=0}^\infty \frac{(-1)^m}{(2m+1)!}P_1(m)t^{2m+1},
 \label{X1X2}
\end{equation}
\begin{equation}
P_1(m)=\prod_{k=1}^m (A-2k+2)(A+2k-1),\quad P_2(m)=\prod_{k=1}^m (A-2k+1)(A+2k).
\label{P1P2}
\end{equation}
The series in (\ref{X1X2}) were proved to be mean fourth convergent for $|t|<1/\e$. Since $X_0,X_1\in\leb^4(\Omega)$, it follows that (\ref{Xt}) is a mean square solution to (\ref{RDE}) on $(-1/\e,1/\e)$.

To summarize, the main result obtained in \cite[Th.~11]{previous} was stated as follows:
\begin{theorem} \label{thElls}
Suppose that $X_0,X_1\in\leb^4(\Omega)$, that $A$ satisfies the growth condition (\ref{EAn}), and that $A$ is independent of $X_0$ and $X_1$. Then the stochastic process defined by (\ref{Xt})--(\ref{P1P2}) is a mean square solution to the random initial value problem (\ref{RDE}) on the time domain $(-1/\e,1/\e)$.
\end{theorem}

Our goal is to extend this theorem and to simplify its proof given in \cite{previous}. 

The growth condition (\ref{EAn}) was established in order to demonstrate the mean fourth convergence of (\ref{X1X2}), by applying well-known inequalities: H\"older's inequality, $c_s$-inequality and arithmetic-geometric inequality. We will simplify the proof given in \cite{previous} by working with an equivalent but easier to manage form of (\ref{EAn}), see Lemma~\ref{equivEAn}. Moreover, the $\leb^\infty(\Omega)$ convergence of (\ref{X1X2}) (which implies mean fourth convergence) will be obtained on the whole interval $(-1,1)$, see Theorem~\ref{th}. This will provide the complete extension of the deterministic counterpart for the random Legendre differential equation.

\begin{lemma}\label{equivEAn}
The growth condition (\ref{EAn}) is equivalent to the boundedness of $A$: $\|A\|_{\leb^{\infty}(\Omega)}<\infty$.
\end{lemma}
\begin{proof}
If $\|A\|_{\leb^{\infty}(\Omega)}<\infty$, then $E[|A|^n]\leq \|A\|_{\leb^{\infty}(\Omega)}^n$, so that we can take $H=1$ and $M=\|A\|_{\leb^{\infty}(\Omega)}$ and (\ref{EAn}) is satisfied.

On the other hand, if (\ref{EAn}) holds, then $\|A\|_{\leb^n(\Omega)}\leq H^{1/n}M$. By taking limits, $\|A\|_{\leb^{\infty}(\Omega)}=\lim_{n\rightarrow \infty} \|A\|_{\leb^{n}(\Omega)}\leq M<\infty$.
\end{proof}

\begin{lemma} \label{absconv}
Let $X(t)=\sum_{n=0}^\infty X_n t^n$ be a formal random power series on $(-1,1)$. Let $1\leq p\leq\infty$. Then the given series converges in $\leb^p(\Omega)$ for all $t\in (-1,1)$, if and only if $\sum_{n=0}^\infty \|X_n\|_{\leb^p(\Omega)}|t|^n<\infty$ for all $t\in (-1,1)$.
\end{lemma}
\begin{proof}
If $\sum_{n=0}^\infty \|X_n\|_{\leb^p(\Omega)}|t|^n<\infty$ for all $t\in (-1,1)$, then the series converges in $\leb^p(\Omega)$ for all $t\in (-1,1)$, because in a Banach space, absolute convergence of a series implies convergence.

On the other hand, suppose that the series converges in $\leb^p(\Omega)$ for all $t\in (-1,1)$. Fix $|t_0|<1$. Let $|t_0|<|\rho|<1$. Since $\sum_{n=0}^\infty \|X_n\|_{\leb^p(\Omega)}|\rho|^n<\infty$, then $\|X_{n}\|_{\leb^p(\Omega)}|\rho|^n\leq 1$, for $n\geq n_0$. Thus, $\|X_{n}\|_{\leb^p(\Omega)}|t_0|^n\leq (|t_0|/|\rho|)^n$, for $n\geq n_0$, with $\sum_{n=0}^{\infty} (|t_0|/|\rho|)^n<\infty$. By comparison, $\sum_{n=0}^\infty \|X_n\|_{\leb^p(\Omega)}|t_0|^n<\infty$.
\end{proof}

We state and proof our main Theorem~\ref{th}. It is a significant improvement of Theorem~(\ref{thElls}) stated and proved in \cite{previous}: for $p=2$, we only require mean square integrability of $X_0$ and $X_1$, not mean fourth integrability; we do not need any independence assumption on $A$, $X_0$ and $X_1$; and we demonstrate mean square convergence of the series on the whole interval $(-1,1)$, not just $(-1/\e,1/\e)$. Moreover, our proof is much simpler, because the hypothesis of boundedness for $A$ instead of the equivalent growth condition (\ref{EAn}) allows simpler and more direct inequalities (we do not need H\"older's inequality, $c_s$-inequality, arithmetic-geometric inequality, etc.).

\begin{theorem} \label{th}
Suppose that $X_0,X_1\in\leb^p(\Omega)$, for certain $1\leq p\leq\infty$, and $\|A\|_{\leb^\infty(\Omega)}<\infty$. Then the stochastic process defined by (\ref{Xt})--(\ref{P1P2}) is the unique $\leb^p(\Omega)$ solution to the random initial value problem (\ref{RDE}) on the whole time domain $(-1,1)$.
\end{theorem}
\begin{proof}
From (\ref{Xt}) and $X_0,X_1\in\leb^p(\Omega)$, it suffices to see that the two series given in (\ref{X1X2}) converge in $\leb^\infty(\Omega)$ for $t\in (-1,1)$. That is,
\begin{equation} 
\sum_{m=0}^\infty \frac{1}{(2m)!}\|P_1(m)\|_{\leb^\infty(\Omega)}|t|^{2m}<\infty,\quad \sum_{m=0}^\infty \frac{1}{(2m+1)!}\|P_1(m)\|_{\leb^\infty(\Omega)}|t|^{2m+1}<\infty, 
 \label{abso}
\end{equation}
for $t\in (-1,1)$ (see Lemma~\ref{absconv}). We will check (\ref{abso}) for the first series, as for the second one the reasoning is completely analogous.

Let $L=\|A\|_{\leb^\infty(\Omega)}$. We have
\small                 
\begin{align*}
 \|P_1(m)\|_{\leb^\infty(\Omega)} = {} &  \left\| \prod_{k=1}^m (A-2k+2)(A+2k-1) \right\|_{\leb^\infty(\Omega)}\leq \prod_{k=1}^m (L+2k-2)(L+2k-1) \\
\leq {} & \prod_{k=1}^m (L+2k-1)^2=\left(\frac{\prod_{k=1}^{2m-1} (L+k)}{\prod_{k=1}^{m-1}(L+2k)}\right)^2=\left(\frac{(L+2m-1)!}{L!\prod_{k=1}^{m-1}(L+2k)}\right)^2 \\
= {} & \left(\frac{(L+2m-1)!}{L!2^{m-1}\prod_{k=1}^{m-1} (L/2+k)}\right)^2=\left(\frac{(L+2m-1)!\Gamma(L/2+1)}{L!2^{m-1}\Gamma(L/2+m)}\right)^2,
\end{align*}
\normalsize
where the property $\Gamma(x)=(x-1)\Gamma(x-1)$ of the Gamma function $\Gamma(z)=\int_0^\infty x^{z-1}\e^{-x}\,\dif x$ has been used. By the root test, if we check that 
\[ \lim_{m\rightarrow\infty} \left(\frac{(L+2m-1)!\Gamma(L/2+1)}{L!2^{m-1}\Gamma(L/2+m)(2m)!^\frac12}\right)^{2/m}=1, \]
then the first part of (\ref{abso}) will follow. By Stirling's formula, $\Gamma(x)\sim \sqrt{2\pi x}(\frac{x-1}{\e})^{x-1}$ as $x\rightarrow\infty$. Then
\begin{align*} 
{} &  \lim_{m\rightarrow\infty} \left(\frac{(L+2m-1)!\Gamma(L/2+1)}{L!2^{m-1}\Gamma(L/2+m)(2m)!^\frac12}\right)^{2/m} \\
= {} &  \lim_{m\rightarrow\infty} \left(\frac{\sqrt{2\pi (L+2m-1)}\left(\frac{L+2m-1}{\e}\right)^{L+2m-1} \Gamma(L/2+1)}{L!2^{m-1}\sqrt{2\pi(L/2+m)}\left(\frac{L/2+m-1}{\e}\right)^{L/2+m-1}\sqrt[4]{4\pi m}\left(\frac{2m}{\e}\right)^m}\right)^{2/m} \\
= {} &  \lim_{m\rightarrow\infty} \frac{\left(\frac{L+2m-1}{\e}\right)^{4}}{4\left(\frac{L/2+m-1}{\e}\right)^2 \left(\frac{2m}{\e}\right)^2}=1.
\end{align*}
As a conclusion, the stochastic process defined by (\ref{Xt})--(\ref{P1P2}) is an $\leb^p(\Omega)$ solution to (\ref{RDE}) on $(-1,1)$.

To demonstrate the uniqueness, we use \cite[Th.~5.1.2]{soong}, \cite[Th.~5]{strand}. Rewrite (\ref{RDE}) as $\dot{Z}(t)=B(t)Z(t)$, where 
\[ Z(t)=\begin{pmatrix} X(t) \\ \dot{X}(t) \end{pmatrix},\quad B(t)=\begin{pmatrix} 0 & 1 \\ \frac{A(A+1)}{1-t^2} & \frac{-2t}{1-t^2} \end{pmatrix}. \]
We say that $Z=(Z_1,Z_2)$ belongs to $\leb^p_2(\Omega)$ if $\|Z\|_{\leb^p_2(\Omega)}:=\max\{\|Z_1\|_{\leb^p(\Omega)},\|Z_2\|_{\leb^p(\Omega)}\}<\infty$. Consider the random matrix norm $|||B|||:=\max_i \sum_j \|b_{ij}\|_{\leb^\infty(\Omega)}$. If $Z,Z'\in\leb^p_2(\Omega)$, then $\|B(t)Z-B(t)Z'\|_{\leb^p_2(\Omega)}\leq |||B(t)|||\cdot\|Z-Z'\|_{\leb^p_2(\Omega)}$, where 
\[ \int_{-a}^a |||B(t)|||\,\dif t=\int_{-a}^a \frac{\|A\|_{\leb^\infty(\Omega)}(\|A\|_{\leb^\infty(\Omega)}+1)+2|t|}{1-t^2}\,\dif t<\infty\] 
for each $a\in (0,1)$. Then the assumptions of \cite[Th.~5.1.2]{soong}, \cite[Th.~5]{strand} hold.
\end{proof}

The hypothesis $\|A\|_{\leb^\infty(\Omega)}<\infty$ is satisfied by some standard probability distributions: Uniform, Beta, Binomial, etc. If one wants $A$ to follow an unbounded distribution, the truncation method permits bounding the support of $A$ (see \cite{trunc}). For example, the truncated Normal or Gamma distributions can be given to $A$. See Example~\ref{ex3} for a test of this methodology.

\begin{remark} \normalfont
If $\|A\|_{\leb^\infty(\Omega)}=\infty$, then (\ref{abso}) does not hold for any $t\in (-1,1)\backslash\{0\}$. Indeed, 
\[ \sum_{m=0}^\infty \frac{1}{(2m)!}\|P_1(m)\|_{\leb^\infty(\Omega)}|t|^{2m}\geq \frac12 \|P_1(1)\|_{\leb^\infty(\Omega)} t^2=\frac12 \|A(A+1)\|_{\leb^\infty(\Omega)} t^2=\infty. \]
By Lemma~\ref{absconv}, the two series given in (\ref{X1X2}) do not converge in $\leb^\infty(\Omega)$, for any $t\in (-1,1)\backslash\{0\}$.
\end{remark}

\begin{remark} \normalfont
If $X_0,X_1,A\in\leb^\infty(\Omega)$, then the response process $X(t)$ defined by (\ref{Xt})--(\ref{P1P2}) is the unique $\leb^\infty(\Omega)$ solution to (\ref{RDE}) on $(-1,1)$. In particular, $X(t)$ is the unique solution in the sample path sense \cite[Appendix~A]{soong}.
\end{remark}

\section{Approximation of the expectation and variance of the response process} \label{approximationEV}

Let $X_0,X_1\in\leb^2(\Omega)$ and $A$ be a bounded random variable, not necessarily independent. By Theorem~\ref{th}, the stochastic process $X(t)$ defined by (\ref{Xt})--(\ref{P1P2}) is an $\leb^2(\Omega)$ solution to the random initial value problem (\ref{RDE}) on the whole time domain $(-1,1)$. If we consider $X^M(t)=X_0 \tilde{X}_1^M(t)+X_1\tilde{X}_2^M(t)$, where
\[ \tilde{X}_1^M(t)=\sum_{m=0}^{\lfloor \frac{M}{2} \rfloor} \frac{(-1)^m}{(2m)!}P_1(m)t^{2m},\quad \tilde{X}_2^M(t)=\sum_{m=0}^{\lfloor \frac{M-1}{2} \rfloor} \frac{(-1)^m}{(2m+1)!}P_2(m)t^{2m+1}, \]
we know that $X^M(t)\rightarrow X(t)$ in $\leb^2(\Omega)$ as $M\rightarrow\infty$, for each $t\in (-1,1)$. This mean square convergence allows us to approximate the expectation and variance of $X(t)$ by using 
\begin{equation}
 \mathbb{E}[X(t)]=\lim_{M\rightarrow\infty} \mathbb{E}[X^M(t)],\quad \mathbb{V}[X(t)]=\lim_{M\rightarrow\infty} \mathbb{V}[X^M(t)], 
 \label{approxEV}
\end{equation}
see \cite[Th.~4.2.1, Th.~4.3.1]{soong}.

The expectation of $X^M(t)$ is given by
\[ \mathbb{E}[X^M(t)]=\sum_{m=0}^{\lfloor \frac{M}{2} \rfloor} \frac{(-1)^m}{(2m)!}\mathbb{E}[X_0P_1(m)]t^{2m}+\sum_{m=0}^{\lfloor \frac{M-1}{2} \rfloor} \frac{(-1)^m}{(2m+1)!}\mathbb{E}[X_1P_2(m)]t^{2m+1}, \]
where 
\[ \mathbb{E}[X_0 P_1(m)]=\int_{(0,\infty)\times\mathbb{R}} x_0\left(\prod_{j=1}^m (a-2j+2)(a+2j-1)\right)\mathbb{P}_{(A,X_0)}(\dif a,\dif x_0), \]
\[ \mathbb{E}[X_1 P_2(m)]=\int_{(0,\infty)\times\mathbb{R}} x_1\left(\prod_{j=1}^m (a-2j+1)(a+2j)\right)\mathbb{P}_{(A,X_1)}(\dif a,\dif x_1). \]
Here, $\mathbb{P}_Z$ represents the probability law of the random vector $Z$, which comprises the different cases of absolute continuity, discrete support, etc.

On the other hand, the variance of $X_M(t)$ is given by
\[ \mathbb{V}[X_M(t)]=\mathbb{E}[X_M(t)^2]-\left(\mathbb{E}[X_M(t)]\right)^2, \]
so that we need to compute $\mathbb{E}[X_M(t)^2]$. Let 
\[ X_{2m}=X_0 \frac{(-1)^m}{(2m)!}P_1(m),\quad X_{2m+1}=X_1 \frac{(-1)^m}{(2m+1)!}P_2(m). \]
We have
\begin{align*} 
\mathbb{E}[X_M(t)^2]= {} &  \mathbb{E}\left[\left(\sum_{m=0}^{\lfloor \frac{M}{2} \rfloor} X_{2m}t^{2m}\right)^2\right]+\mathbb{E}\left[\left(\sum_{m=0}^{\lfloor \frac{M-1}{2} \rfloor} X_{2m+1}t^{2m+1}\right)^2\right] \\
+ {} &  2\sum_{m=0}^{\lfloor \frac{M}{2} \rfloor}\sum_{n=0}^{\lfloor \frac{M-1}{2} \rfloor} \mathbb{E}[X_{2m}X_{2n+1}]t^{2(m+n)+1}, 
\end{align*}
where 
\[ \mathbb{E}\left[\left(\sum_{m=0}^{\lfloor \frac{M}{2} \rfloor} X_{2m}t^{2m}\right)^2\right]=\sum_{m=0}^{\lfloor \frac{M}{2} \rfloor} \sum_{n=0}^{\lfloor \frac{M}{2} \rfloor} \mathbb{E}[X_{2m}X_{2n}]t^{2(m+n)}, \]
\[ \mathbb{E}\left[\left(\sum_{m=0}^{\lfloor \frac{M-1}{2} \rfloor} X_{2m+1}t^{2m+1}\right)^2\right]=\sum_{m=0}^{\lfloor \frac{M-1}{2} \rfloor}\sum_{n=0}^{\lfloor \frac{M-1}{2} \rfloor}\mathbb{E}[X_{2m+1}X_{2n+1}]t^{2(m+n)+2}. \]
The expectations involved in these expressions can be computed as follows:
\small
\begin{align*} 
\mathbb{E}[X_{2m}X_{2n}]= {} &  \frac{(-1)^{m+n}}{(2m)!(2n)!}\mathbb{E}[X_0^2 P_1(m)P_1(n)] \\
= {} &  \frac{(-1)^{m+n}}{(2m)!(2n)!}\int_{(0,\infty)\times\mathbb{R}} x_0^2 \left(\prod_{j=1}^m (a-2j+2)(a+2j-1)\right) \\
\cdot {} &  \left(\prod_{j=1}^n (a-2j+2)(a+2j-1)\right)\mathbb{P}_{(A,X_0)}(\dif a,\dif x_0), 
\end{align*}
\begin{align*} 
\mathbb{E}[X_{2m+1}X_{2n+1}]= {} &  \frac{(-1)^{m+n}}{(2m+1)!(2n+1)!}\mathbb{E}[X_1^2 P_2(m)P_2(n)] \\
= {} &  \frac{(-1)^{m+n}}{(2m+1)!(2n+1)!}\int_{(0,\infty)\times\mathbb{R}} x_1^2 \left(\prod_{j=1}^m (a-2j+1)(a+2j)\right) \\
\cdot {} &  \left(\prod_{j=1}^n (a-2j+1)(a+2j)\right)\mathbb{P}_{(A,X_1)}(\dif a,\dif x_1), 
\end{align*}
\begin{align*} 
\mathbb{E}[X_{2m}X_{2n+1}]= {} &  \frac{(-1)^{m+n}}{(2m)!(2n+1)!}\mathbb{E}[X_0 X_1 P_1(m)P_2(n)] \\
= {} &  \frac{(-1)^{m+n}}{(2m)!(2n)!}\int_{(0,\infty)\times\mathbb{R}\times\mathbb{R}} x_0 x_1 \left(\prod_{j=1}^m (a-2j+2)(a+2j-1)\right) \\
\cdot {} &  \left(\prod_{j=1}^n (a-2j+1)(a+2j)\right)\mathbb{P}_{(A,X_0,X_1)}(\dif a,\dif x_0,\dif x_1). 
\end{align*}
\normalsize

\section{Numerical experiments} \label{experiments}

In this section we perform several numerical experiments. Since in \cite{previous} the authors carried out numerical examples when $A$, $X_0$ and $X_1$ are independent random variables, we will show three more examples in which $A$, $X_0$ and $X_1$ are not independent. To assess the reliability of the approximations obtained for the expectation and variance by using (\ref{approxEV}), we will compare them with Monte Carlo simulations and a generalized Polynomial Chaos (gPC) approach. 

Monte Carlo simulations generate samples of $X(t)$ by computing realizations of $A$, $X_0$ and $X_1$ and solving the corresponding deterministic problem (\ref{RDE}). Although it is an effective and easy to implement approach to quantify the uncertainty, the slowness to get accurately the digits in the computations makes this technique computationally expensive, \cite{MC}, \cite[pp.~53--54]{xiu}.

Our gPC approach is based on the computational algorithm presented in \cite{depen}, which works when the random input parameters are non-independent and jointly absolutely continuous. Due to the spectral convergence of the Galerkin projections in $\leb^2(\Omega)$ \cite{xiu,ernst,shi,nostre}, for small orders of bases $m$ (see \cite{depen}) the approximations for the expectation and variance are very accurate, especially for small $t$. However, increasing the order $m$ of the bases may entail numerical errors, see \cite{shi,nostre} and Example~\ref{ex3}.

\begin{example} \label{ex1} \normalfont

We consider the random differential equation (\ref{RDE}) with 
\[ (A,X_0,X_1)\sim\text{Dirichlet}(5,1,2,3). \]
Since $X_0$, $X_1$ and $A$ are bounded random variables, Theorem~\ref{th} implies that the stochastic process $X(t)$ defined by (\ref{Xt})--(\ref{P1P2}) is the unique $\leb^\infty(\Omega)$ solution to (\ref{RDE}) on $(-1,1)$. In Table~\ref{exp1}, we show $\mathbb{E}[X^M(t)]$ for different orders $M$, which approximates $\mathbb{E}[X(t)]$ by (\ref{approxEV}). We observe that the approximations achieved are more accurate for small $M$ when $t$ is near $0$, because the random power series is centered at $0$ and the process $X(t)$ is known at $0$. For $t\leq 0.8$, stabilization of the results has been achieved for $M=80$. For $t=0.9$, a larger $M$ would be needed. We notice that Monte Carlo simulations with $500,000$ realizations give an approximate result up to three significant figures. To obtain more exact approximations, more simulations and computational cost are needed. In general, the approximations via Monte Carlo simulations are worse than via our Fr\"obenius method. Concerning gPC approximations, the results obtained are as accurate as via the Fr\"obenius method, due to its spectral convergence. Table~\ref{var2} provides analogous results for the variance, where $\mathbb{V}[X^M(t)]$ approximates $\mathbb{V}[X(t)]$ by (\ref{approxEV}). For $t\leq 0.7$ stabilization of the approximations has been reached for $M=80$. In general, a larger $M$ is required to achieve nearly exact approximations for the variance. The results obtained agree with Monte Carlo simulations and gPC expansions.

\begin{table}[hbtp!]
\begin{center}
\begin{tabular}{|c|c|c|c|c|c|c|} \hline
$t$ & $\mathbb{E}[X^{10}(t)]$ & $\mathbb{E}[X^{20}(t)]$ & $\mathbb{E}[X^{40}(t)]$ & $\mathbb{E}[X^{80}(t)]$ & MC $500,000$ & gPC $m=3$  \\ \hline
$0$ & $0.0909091$ & $0.0909091$ & $0.0909091$ & $0.0909091$ & $0.0906787$ & $0.0909091$ \\ \hline
$0.1$ & $0.108855$ & $0.108855$ & $0.108855$ & $0.108855$ & $0.108648$ & $0.108855$ \\ \hline
$0.2$ & $0.126491$ & $0.126491$ & $0.126491$ & $0.126491$ & $0.126308$ & $0.126491$ \\ \hline
$0.3$ & $0.144059$ & $0.144059$ & $0.144059$ & $0.144059$ & $0.143903$ & $0.144059$ \\ \hline
$0.4$ & $0.161835$ & $0.161835$ & $0.161835$ & $0.161835$ & $0.161709$ & $0.161835$ \\ \hline
$0.5$ & $0.180166$ & $0.180172$ & $0.180172$ & $0.180172$ & $0.180080$ & $0.180172$ \\ \hline
$0.6$ & $0.199548$ & $0.199591$ & $0.199592$ & $0.199592$ & $0.199540$ & $0.199592$ \\ \hline
$0.7$ & $0.220733$ & $0.220998$ & $0.221002$ & $0.221002$ & $0.221000$ & $0.221002$ \\ \hline
$0.8$ & $0.24491$ & $0.246266$ & $0.246352$ & $0.246352$ & $0.246416$ & $0.246352$ \\ \hline
$0.9$ & $0.273962$ & $0.280100$ & $0.281585$ & $0.281693$ & $0.281863$ & $0.281694$ \\ \hline
\end{tabular}
\caption{Approximation of the expectation of the solution stochastic process. Example \ref{ex1}.}
\label{exp1}
\end{center}
\end{table}

\begin{table}[hbtp!]
\begin{center}
\begin{tabular}{|c|c|c|c|c|c|c|} \hline
$t$ & $\mathbb{V}[X^{10}(t)]$ & $\mathbb{V}[X^{20}(t)]$ & $\mathbb{V}[X^{40}(t)]$ & $\mathbb{V}[X^{80}(t)]$ & MC $500,000$ & gPC $m=3$  \\ \hline
$0$ & $0.00688705$ & $0.00688705$ & $0.00688705$ & $0.00688705$ & $0.00685105$ & $0.00688705$ \\ \hline
$0.1$ & $0.00670461$ & $0.00670461$ & $0.00670461$ & $0.00670461$ & $0.00666882$ & $0.00670461$ \\ \hline
$0.2$ & $0.00672130$ & $0.00672130$ & $0.00672130$ & $0.00672130$ & $0.00668621$ & $0.00672130$ \\ \hline
$0.3$ & $0.00697044$ & $0.00697045$ & $0.00697045$ & $0.00697045$ & $0.00693658$ & $0.00697045$ \\ \hline
$0.4$ & $0.00751088$ & $0.00751091$ & $0.00751091$ & $0.00751091$ & $0.00747887$ & $0.00751091$ \\ \hline
$0.5$ & $0.00844437$ & $0.00844482$ & $0.00844482$ & $0.00844482$ & $0.00841536$ & $0.00844482$ \\ \hline
$0.6$ & $0.00995308$ & $0.00995823$ & $0.00995825$ & $0.00995825$ & $0.00993237$ & $0.00995825$ \\ \hline
$0.7$ & $0.0123829$ & $0.0124269$ & $0.0124276$ & $0.0124276$ & $0.0124068$ & $0.0124276$ \\ \hline
$0.8$ & $0.0164346$ & $0.0167508$ & $0.0167712$ & $0.0167714$ & $0.0167582$ & $0.0167714$ \\ \hline
$0.9$ & $0.0236175$ & $0.0256974$ & $0.0262304$ & $0.0262699$ & $0.0262712$ & $0.0262702$ \\ \hline
\end{tabular}
\caption{Approximation of the variance of the solution stochastic process. Example \ref{ex1}.}
\label{var1}
\end{center}
\end{table}

\end{example}

\begin{example} \label{ex2} \normalfont

We set a joint discrete distribution to $(A,X_0,X_1)$:
\[ (A,X_0,X_1)\sim \text{Multinomial}(10;0.2,0.3,0.5). \]
Since $X_0$, $X_1$ and $A$ are bounded random variables, Theorem~\ref{th} entails that the response process $X(t)$ defined by (\ref{Xt})--(\ref{P1P2}) is the unique $\leb^\infty(\Omega)$ solution to (\ref{RDE}) on $(-1,1)$. Expression (\ref{approxEV}) allows approximating $\mathbb{E}[X(t)]$ and $\mathbb{V}[X(t)]$ via $\mathbb{E}[X^M(t)]$ and $\mathbb{V}[X^M(t)]$, respectively. Analogous comments to the previous example apply here, and the results are presented in Table~\ref{exp2} and Table~\ref{var2}. We point out that, since $(A,X_0,X_1)$ is discrete, the computational method from \cite{depen} to apply gPC expansions does not work in this case. The results obtained via our Fr\"obenius method are accurate. 

\begin{table}[hbtp!]
\begin{center}
\begin{tabular}{|c|c|c|c|c|c|} \hline
$t$ & $\mathbb{E}[X^{10}(t)]$ & $\mathbb{E}[X^{20}(t)]$ & $\mathbb{E}[X^{40}(t)]$ & $\mathbb{E}[X^{80}(t)]$ & MC $500,000$  \\ \hline
$0$ & $3$ & $3$ & $3$ & $3$ & $3.00207$ \\ \hline
$0.1$ & $3.39965$ & $3.39965$ & $3.39965$ & $3.39965$ & $3.40154$ \\ \hline
$0.2$ & $3.59067$ & $3.59067$ & $3.59067$ & $3.59067$ & $3.59226$ \\ \hline
$0.3$ & $3.57194$ & $3.57194$ & $3.57194$ & $3.57194$ & $3.57322$ \\ \hline
$0.4$ & $3.35661$ & $3.35661$ & $3.35661$ & $3.35661$ & $3.35768$ \\ \hline
$0.5$ & $2.97154$ & $2.97154$ & $2.97154$ & $2.97154$ & $2.97259$ \\ \hline
$0.6$ & $2.45625$ & $2.45623$ & $2.45623$ & $2.45623$ & $2.45738$ \\ \hline
$0.7$ & $1.86122$ & $1.86112$ & $1.86111$ & $1.86111$ & $1.86226$ \\ \hline
$0.8$ & $1.24584$ & $1.24523$ & $1.24515$ & $1.24515$ & $1.24558$ \\ \hline
$0.9$ & $0.675881$ & $0.672543$ & $0.670722$ & $0.670550$ & $0.668041$ \\ \hline
\end{tabular}
\caption{Approximation of the expectation of the solution stochastic process. Example \ref{ex2}.}
\label{exp2}
\end{center}
\end{table}

\begin{table}[hbtp!]
\begin{center}
\begin{tabular}{|c|c|c|c|c|c|} \hline
$t$ & $\mathbb{V}[X^{10}(t)]$ & $\mathbb{V}[X^{20}(t)]$ & $\mathbb{V}[X^{40}(t)]$ & $\mathbb{V}[X^{80}(t)]$ & MC $500,000$  \\ \hline
$0$ & $2.1$ & $2.1$ & $2.1$ & $2.1$ & $2.10094$ \\ \hline
$0.1$ & $1.81331$ & $1.81331$ & $1.81331$ & $1.81331$ & $1.81300$ \\ \hline
$0.2$ & $1.78089$ & $1.78089$ & $1.78089$ & $1.78089$ & $1.77973$ \\ \hline
$0.3$ & $2.43304$ & $2.43304$ & $2.43304$ & $2.43304$ & $2.43226$ \\ \hline
$0.4$ & $4.15996$ & $4.15996$ & $4.15996$ & $4.15996$ & $4.16027$ \\ \hline
$0.5$ & $7.12080$ & $7.12077$ & $7.12077$ & $7.12077$ & $7.12084$ \\ \hline
$0.6$ & $11.1844$ & $11.1838$ & $11.1838$ & $11.1838$ & $11.1812$ \\ \hline
$0.7$ & $16.1150$ & $16.1090$ & $16.1090$ & $16.1090$ & $16.1046$ \\ \hline
$0.8$ & $22.1109$ & $22.0920$ & $22.0932$ & $22.0932$ & $22.0965$ \\ \hline
$0.9$ & $31.1044$ & $31.4557$ & $31.6189$ & $31.6324$ & $31.6569$ \\ \hline
\end{tabular}
\caption{Approximation of the variance of the solution stochastic process. Example \ref{ex2}.}
\label{var2}
\end{center}
\end{table}

\end{example}

\begin{example} \label{ex3} \normalfont

We set a truncated Multinormal distribution for the random input parameters:
\[ (A,X_0,X_1)\sim\text{Multinormal}(\begin{pmatrix} 10 \\ -2 \\ 1 \end{pmatrix},\begin{pmatrix} 1 & 0.01 & -0.02 \\ 0.01 & 4 & 2 \\ -0.02 & 2 & 4 \end{pmatrix})|_{[6,14]\times\mathbb{R}\times\mathbb{R}}. \]
Since $X_0,X_1\in\leb^p(\Omega)$ for all $1\leq p<\infty$ and $A$ is bounded in $[6,14]$, Theorem~\ref{th} shows that the stochastic process $X(t)$ defined by (\ref{Xt})--(\ref{P1P2}) is the unique $\leb^p(\Omega)$ solution to (\ref{RDE}) on $(-1,1)$, for each $1\leq p<\infty$. Analogously to the previous two examples, Table~\ref{exp3} and Table~\ref{var3} show the results. Observe that stabilization of the results for $t\leq 0.7$ is achieved for $M=80$. Notice also that, for $M\leq 20$ and $t\geq0.4$, the approximation of the expectation and variance is not good. The results obtained from the Fr\"obenius method for $M\geq 80$ agree with the statistics calculated via Monte Carlo simulations. On the other hand, the approximations performed by gPC expansions are not good. This is due to the accumulation of numerical errors, which invalidates the corresponding results. See \cite{shi,nostre} for an analysis of computational errors when working with gPC expansions. Thus, the Fr\"obenius method proves to be the best uncertainty quantification technique for this example.

\begin{table}[hbtp!]
\footnotesize
\tabcolsep=0.07cm
\begin{center}
\begin{tabular}{|c|c|c|c|c|c|c|c|} \hline
$t$ & $\mathbb{E}[X^{10}(t)]$ & $\mathbb{E}[X^{20}(t)]$ & $\mathbb{E}[X^{40}(t)]$ & $\mathbb{E}[X^{80}(t)]$ & MC $500,000$ & gPC $m=3$ & gPC $m=4$  \\ \hline
$0$ & $-2.01642$ & $-2.01642$ & $-2.01642$ & $-2.01642$ & $-2.00100$ & $-2.01642$ & $-2.01642$ \\ \hline
$0.1$ & $-0.905676$ & $-0.905676$ & $-0.905676$ & $-0.905676$ & $-0.905209$ & $-0.917349$ & $-0.91765$ \\ \hline
$0.2$ & $1.10885$ & $1.10884$ & $1.10884$ & $1.10884$ & $1.11031$ & $1.10512$ & $1.10054$ \\ \hline
$0.3$ & $1.94955$ & $1.94909$ & $1.94909$ & $1.94909$ & $1.94966$ & $1.94172$ & $1.93194$ \\ \hline
$0.4$ & $0.656784$ & $0.643176$ & $0.643176$ & $0.643176$ & $0.641893$ & $0.616062$ & $1.26544$ \\ \hline
$0.5$ & $-1.20831$ & $-1.39804$ & $-1.39804$ & $-1.39804$ & $-1.39941$ & $-1.34976$ & $16.2519$ \\ \hline
$0.6$ & $0.111123$ & $-1.57901$ & $-1.57903$ & $-1.57903$ & $-1.57838$ & $-0.565463$ & $286.226$ \\ \hline
$0.7$ & $10.8410$ & $0.602665$ & $0.602084$ & $0.602084$ & $0.594087$ & $8.20453$ & $1524.28$ \\ \hline
$0.8$ & $51.7915$ & $1.58890$ & $1.57617$ & $1.57615$ & $1.57588$ & $50.6284$ & $-211410$ \\ \hline
$0.9$ & $203.700$ & $-0.987211$ & $-1.20468$ & $-1.20776$ & $-1.20091$ & $291.704$ & $-2.51516\cdot 10^7$ \\ \hline
\end{tabular}
\caption{Approximation of the expectation of the solution stochastic process. Example \ref{ex3}.}
\label{exp3}
\end{center}
\end{table}

\begin{table}[hbtp!]
\footnotesize
\tabcolsep=0.07cm
\begin{center}
\begin{tabular}{|c|c|c|c|c|c|c|c|} \hline
$t$ & $\mathbb{V}[X^{10}(t)]$ & $\mathbb{V}[X^{20}(t)]$ & $\mathbb{V}[X^{40}(t)]$ & $\mathbb{V}[X^{80}(t)]$ & MC $500,000$ & gPC $m=3$ & gPC $m=4$  \\ \hline
$0$ & $3.96931$ & $3.96931$ & $3.96931$ & $3.96931$ & $4.00268$ & $3.96931$ & $3.96931$ \\ \hline
$0.1$ & $1.23016$ & $1.23016$ & $1.23016$ & $1.23016$ & $1.22715$ & $1.17839$ & $1.16808$ \\ \hline
$0.2$ & $1.16167$ & $1.16166$ & $1.16166$ & $1.16166$ & $1.14804$ & $0.816119$ & $-7.83909$ \\ \hline
$0.3$ & $3.86797$ & $3.87079$ & $3.87079$ & $3.87079$ & $3.86348$ & $-1.71661$ & $-494.045$ \\ \hline
$0.4$ & $1.72091$ & $1.76984$ & $1.76984$ & $1.76984$ & $1.76343$ & $-54.3357$ & $156807$ \\ \hline
$0.5$ & $2.59759$ & $2.75802$ & $2.75802$ & $2.75802$ & $2.71796$ & $-187.680$ & $1.98436\cdot 10^7$ \\ \hline
$0.6$ & $53.8179$ & $3.79667$ & $3.79665$ & $3.79665$ & $3.79030$ & $7387.73$ & $-9.73065\cdot 10^9$ \\ \hline
$0.7$ & $1774.74$ & $3.88379$ & $3.87941$ & $3.87941$ & $3.88103$ & $244717$ & $-1.46846\cdot 10^{12}$ \\ \hline
$0.8$ & $40373.8$ & $5.25517$ & $5.27273$ & $5.27282$ & $5.17336$ & $2.49059\cdot 10^6$ & $8.46091\cdot 10^{15}$ \\ \hline
$0.9$ & $658630$ & $4.79558$ & $7.67724$ & $7.76726$ & $7.73295$ & $-7.31059\cdot 10^8$ & $-9.69951\cdot 10^{19}$ \\ \hline
\end{tabular}
\caption{Approximation of the variance of the solution stochastic process. Example \ref{ex3}.}
\label{var3}
\end{center}
\end{table}

\end{example}

\section{Conclusions} \label{conclusions}

In this article we have studied the random Legendre differential equation with input coefficient $A$ and initial conditions $X_0$ and $X_1$. In [Calbo G. et al, Comput. Math. Appl., 61(9), 2782--2792 (2011)], a mean square convergent random power series solution $X(t)$ on $(-1/\e,1/\e)$ was constructed via the Fr\"obenius method. The authors proved that, under the assumption that the absolute moments of $A$ grow at most exponentially, under mean fourth integrability of $X_0$ and $X_1$, and under independence of $A$ and the initial conditions, the random power series becomes a mean square solution to the random Legendre differential equation on $(-1/\e,1/\e)$. We have extended this result by assuming less integrability of $X_0$ and $X_1$ and no independence between the random inputs. Moreover, the growth condition on the absolute moments of $A$ has been characterized in terms of the boundedness of $A$. This has permitted a simpler proof of our result, as no probabilistic inequalities (H\"older, $c_s$, etc.) have been required. Moreover, our random power series solution converges on the whole $(-1,1)$, as it occurs with its deterministic counterpart. We have provided expressions for the approximate expectation and variance of $X(t)$, by truncating the random power series. In the numerical examples, we have illustrated the improvements developed by working with non-independent random inputs. Our approach has improved the approximations done by Monte Carlo simulations and gPC expansions.

\section*{Acknowledgements}
This work has been supported by the Spanish Ministerio de Econom\'{i}a y Competitividad grant MTM2017--89664--P. Marc Jornet acknowledges the doctorate scholarship granted by Programa de Ayudas de Investigaci\'on y Desarrollo (PAID), Universitat Polit\`ecnica de Val\`encia.

\section*{Conflict of Interest Statement} 
The authors declare that there is no conflict of interests regarding the publication of this article.


\begin{thebibliography}{10}

\bibitem{soong}
T. T. Soong. \textit{Random Differential Equations in Science and Engineering}. Academic Press, New York, 1973.

\bibitem{strand}
J. L. Strand.  \textit{Random ordinary differential equations}. Journal of Differential Equations, 7(3)  (1970), 538--553.

\bibitem{smith}
R. C. Smith. \textit{Uncertainty Quantification. Theory, Implementation, and Application}. SIAM Computational Science \& Engineering, 2014.

\bibitem{MC}
G. Fishman. \textit{Monte Carlo: Concepts, Algorithms, and Applications}. Springer Science \& Business Media, 2013.

\bibitem{depen}
J.-C. Cort\'es, J.-V. Romero, M.-D. Rosell\'o, F.-J. Santonja, R.-J. Villanueva. \textit{Solving continuous models with dependent uncertainty: A computational approach}. Abstract and Applied Analysis, 2013 (2013).

\bibitem{xiu}
D. Xiu. \textit{Numerical Methods for Stochastic Computations. A Spectral Method Approach}. Cambridge Texts in Applied Mathematics, Princeton, University Press, New York, 2010.

\bibitem{tawil}
M. A. El-Tawil. \textit{The approximate solutions of some stochastic differential equations using transformations}. Applied Mathematics and Computation, 164(1) (2005), 167--178.

\bibitem{diffusion}
J.-C. Cort\'es, P. Sevilla-Peris, L. J\'odar.  \textit{Constructing approximate diffusion processes with uncertain data}. Mathematics and Computers in Simulation, 73(1--4) (2006), 125--132.

\bibitem{diffusion2}
J.-C. Cort\'es, L. J\'odar, L. Villafuerte, R.-J. Villanueva. \textit{Computing mean square approximations of random diffusion models with source term}. Mathematics and Computers in Simulation, 76(1--3) (2007), 44--48.

\bibitem{Khodabin_MCM_2011}
M. Khodabin, K. Maleknejad, M. Rostami, M. Nouri. \textit{Numerical solution of stochastic differential equations by second order Runge-Kutta methods}. Mathematical and Computer Modelling, 53(9--10) (2011), 1910--1920.

\bibitem{Nouri_Mediterr_2015}
K. Nouri, H. Ranjbar. \textit{Mean square convergence of the numerical solution of random differential equations}. Mediterranean Journal of Mathematics, 12(3) (2015), 1123--1140.

\bibitem{Nouri_Mediterr_2016}
K. Nouri. \textit{Study on stochastic differential equations via modified Adomian decomposition method}. U.P.B. Sci. Bull., Ser. A, 78(1) (2016), 81--90.

\bibitem{Khodabin_AD_2015}
M. Khodabin, M. Rostami. \textit{Mean square numerical solution of stochastic differential equations by fourth order Runge-Kutta method and its application in the electric circuits with noise}. Advances in Difference Equations, 623 (2015), 1--19.

\bibitem{Jerez_Steklov}
S. D\'{i}az-Infante, S. Jerez. \textit{Convergence and asymptotic stability of the explicit Steklov method for stochastic differential equations}. Journal of Computational and Applied Mathematics, 291(1)  (2016), 36--47.   

\bibitem{Soleymani_2015}
Ali~R.~Soheili, F.~Toutounian, F.~Soleymani. \textit{A fast convergent numerical method for matrix sign function with application in SDEs (Stochastic Differential Equations)}. Journal of Computational and Applied Mathematics 282 (2015), 167--178. 

\bibitem{oksendal}
{\O}ksendal B. \textit{Stochastic Differential Equations}. Springer, 2003.

\bibitem{L4jc}
L. Villafuerte, C. A. Braumann, J.-C. Cort\'es, L. J\'odar. \textit{Random differential operational calculus: theory and applications}. Computers \& Mathematics with Applications, 59(1) (2010), 115--125.

\bibitem{Licea_JCAM}
J.~Licea, L.~Villafuerte, B.~M.~Chen-Charpentier. \textit{Analytic and numerical solutions of a Riccati differential equation with random coefficients}. Journal of Computational and Applied Mathematics, 309(1) (2013), 208--219.  

\bibitem{airy}
J.-C. Cort\'es, L. J\'odar, J. Camacho, L. Villafuerte. \textit{Random Airy type differential equations: Mean square exact and numerical solutions}. Computers and Mathematics with Applications, 60(5) (2010), 1237--1244. 

\bibitem{hermite}
G. Calbo, J.-C. Cort\'es, L. J\'odar. \textit{Random Hermite differential equations: Mean square power series solutions and statistical properties}. Applied Mathematics and Computation, 218(7) (2011), 3654--3666. 

\bibitem{previous}
G. Calbo, J.-C. Cort\'es, L. J\'odar, L. Villafuerte. \textit{Solving the random Legendre differential equation: Mean square power series solution and its statistical functions}. Computers \& Mathematics with Applications, 61(9) (2011), 2782--2792.

\bibitem{bessel}
J.~C.~Cort\'es, L.~J\'odar,  L.~Villafuerte. \textit{Mean square solution of Bessel differential equation with uncertainties}. Journal of Computational and Applied Mathematics 309(1) (2017), 383--395.  

\bibitem{homotopy_Romanian} 
A. K. Golmankhaneh, N. A. Porghoveh, D. Baleanu. \textit{Mean square solutions of second-order random differential equations by using homotopy analysis method}. Romanian Reports in Physics, 65(2) (2013), 350--362.

\bibitem{Khudair_1}
A.~K.~Khudair, A.~A.~Ameen, S.~L.~Khalaf. \textit{Mean square solutions of second-order random differential equations by using Adomian decomposition method}. Applied Mathematical Sciences 51(5) (2011), 2521--2535.   

\bibitem{Khudair_2}
A.~K.~Khudair, S.~A.~M.~Haddad, S.~L.~Khalaf. \textit{Mean square solutions of second-order random differential equations by using the differential transformation method}. Open Journal of Applied Sciences 6 (2016), 287--297.  

\bibitem{trunc}
L. Norman, S. Kotz, N. Balakrishnan. \textit{Continuous Univariate Distributions}. Volume~1, Wiley, 1994.

\bibitem{ernst}
O. G. Ernst, A. Mugler, H.-J. Starkloff, E. Ullmann. \textit{On the convergence of generalized polynomial chaos expansions}, ESAIM: Math. Modell. Num. Anal., 46(2) (2012), 317--339.

\bibitem{shi}
W. Shi, C. Zhang. \textit{Error analysis of generalized polynomial chaos for nonlinear random ordinary differential equations}. Appl. Num. Math., 62(12) (2012), 1954--1964.

\bibitem{nostre}
J. Calatayud, J.-C. Cort\'es, M. Jornet. \textit{On the convergence of adaptive gPC for non-linear random difference equations: Theoretical analysis and some practical recommendations}. J. Nonlinear Sci. Appl., 11(9) (2018), 1077--1084.
%


\end{thebibliography}
\end{document}